\documentclass[a4paper,12pt,reqno]{amsart}
\usepackage{amsfonts}
\usepackage{amsmath}
\usepackage{amssymb}
\usepackage[a4paper]{geometry}
\usepackage{mathrsfs}
\usepackage{hyperref}
\renewcommand\eqref[1]{(\ref{#1})} 
\numberwithin{equation}{section}
\theoremstyle{plain}
\newtheorem{thm}{Theorem}[section]
\newtheorem{prop}[thm]{Proposition}
\newtheorem{cor}[thm]{Corollary}

\theoremstyle{definition}

\newtheorem{rem}[thm]{Remark}

\begin{document}

   \title[Horizontal Hardy, Rellich and Caffarelli-Kohn-Nirenberg inequalities]
   {On horizontal Hardy, Rellich, Caffarelli-Kohn-Nirenberg and $p$-sub-Laplacian inequalities on stratified groups}

\author[Michael Ruzhansky]{Michael Ruzhansky}
\address{
  Michael Ruzhansky:
  \endgraf
  Department of Mathematics
  \endgraf
  Imperial College London
  \endgraf
  180 Queen's Gate, London SW7 2AZ
  \endgraf
  United Kingdom
  \endgraf
  {\it E-mail address} {\rm m.ruzhansky@imperial.ac.uk}
  }
\author[Durvudkhan Suragan]{Durvudkhan Suragan}
\address{
  Durvudkhan Suragan:
  \endgraf
  Institute of Mathematics and Mathematical Modelling
  \endgraf
  125 Pushkin str.
  \endgraf
  050010 Almaty
  \endgraf
  Kazakhstan
  \endgraf
  and
  \endgraf
  Department of Mathematics
  \endgraf
  Imperial College London
  \endgraf
  180 Queen's Gate, London SW7 2AZ
  \endgraf
  United Kingdom
  \endgraf
  {\it E-mail address} {\rm d.suragan@imperial.ac.uk}
  }

\thanks{The authors were supported in parts by the EPSRC
 grant EP/K039407/1 and by the Leverhulme Grant RPG-2014-02,
 as well as by the MESRK grant 5127/GF4.}


     \begin{abstract}
     In this paper, we present a version of horizontal weighted Hardy-Rellich type and Caffarelli-Kohn-Nirenberg type inequalities on stratified groups and study some of their consequences. Our results reflect on many  results previously known in special cases. Moreover, a new simple proof of the Badiale-Tarantello conjecture \cite{BadTar:ARMA-2002} on the best constant of a Hardy type inequality is provided. We also show a family of Poincar\'e inequalities as well as inequalities involving the weighted and unweighted $p$-sub-Laplacians.
     \end{abstract}
     \maketitle

\section{Introduction}

Consider the following inequality
\begin{equation}\label{HRn-p}
\left\|\frac{f(x)}{\|x\|}\right\|_{L^{p}(\mathbb{R}^{n})} \leq \frac{p}{n-p}\left\| \nabla f\right\|_{L^{p}(\mathbb{R}^{n})},\quad 1\leq p<n,
\end{equation}
where $\nabla$ is the standard gradient in
$\mathbb{R}^{n}$, $f\in C_{0}^{\infty}(\mathbb{R}^{n}
\backslash \{0\})$, $\|x\|=\sqrt{x_{1}^{2}+...+x_{n}^{2}},$
and the constant $\frac{p}{n-p}$ is known to be sharp.
The one-dimensional version of \eqref{HRn-p} for $p=2$ was first discovered by
Hardy in \cite{Hardy1919}, and then for other $p$ in \cite{Hardy:1920}, see also \cite{Hardy:1920} for the story behind these inequalities. Since then the inequality \eqref{HRn-p} has been widely analysed in many different settings (see e.g. \cite{Adimurthi-Sekar}-\cite{Costa-2008}, \cite{DAmbrosio-Hardy}, \cite{DAmbrosio-Dipierro:Hardy-manifolds-AHP-2014}, \cite{Davies}, \cite{Davies-Hinz}, \cite{EKL:Hardy-p-Lap}, \cite{HHLT-Hardy-many-particles}, \cite{Laptev15}). Nowadays there is vast literature on this subject, for example, the MathSciNet search shows about 5000 research works related to this topic.
On homogeneous Carnot groups (or stratified groups) inequalities of this type have been also intensively investigated (see e.g. \cite{DGP-Hardy-potanal}, \cite{GL}, \cite{GolKom}, \cite{Grillo:Hardy-Rellich-PA-2003}, \cite{Jin-Shen:Hardy-Rellich-AM-2011}, \cite{Kogoj-Sonner:Hardy-lambda-CVEE-2016}, \cite{Kombe-Ozaydin:Hardy-Rellich-mfds},  \cite{Lian:Rellich}, \cite{NZW-Hardy-p}). In this case inequality \eqref{HRn-p} takes the form
\begin{equation}\label{HG-p}
\left\|\frac{f(x)}{d(x)}\right\|_{L^{p}(\mathbb G)} \leq \frac{p}{Q-p}\left\| \nabla_{H} f\right\|_{L^{p}(\mathbb G)},\quad Q\geq 3,\; 1<p<Q,
\end{equation}
where $Q$ is the homogeneous dimension of the stratified group $\mathbb G$, $\nabla_{H}$ is the horizontal gradient, and $d(x)$ is the so-called $\mathcal L$-gauge, which is a particular homogeneous quasi-norm obtained from the fundamental solution of the sub-Laplacian, that is, $d(x)^{2-Q}$ is a constant multiple of Folland's \cite{Folland-FS} (see also \cite{FS-Hardy}) fundamental solution of the sub-Laplacian on $\mathbb G$. For a short review in this direction and some further discussions we refer to our recent papers \cite{Ruzhansky-Suragan:Layers, Ruzhansky-Suragan:critical, Ruzhansky-Suragan:Hardy,Ruzhansky-Suragan:identities, Ruzhansky-Suragan:uncertainty} and \cite{ORS16} as well as to references therein.

The main aim of this paper is to give analogues of Hardy type inequalities on stratified groups with horizontal gradients and weights. Actually we obtain more than that, i.e., we prove general (horizontal) weighted Hardy, Rellich and Caffarelli-Kohn-Nirenberg type inequalities on stratified groups. Our results extend known Hardy type inequalities on abelian and Heisenberg groups, for example (see e.g. \cite{BadTar:ARMA-2002} and \cite{DAmbrosio-Difur04}).
For the convenience of the reader let us now briefly recapture the main results of this paper.
Let $\mathbb{G}$ be a homogeneous stratified group of homogeneous dimension
$Q$, and let $X_{1},\ldots,X_{N}$ be left-invariant vector fields giving the first
stratum of the Lie algebra of $\mathbb{G}$, $\nabla_{H}=(X_{1},\ldots,X_{N})$, with the sub-Laplacian
\begin{equation*}\label{sublap0}
\mathcal{L}=\sum_{k=1}^{N}X_{k}^{2}.
\end{equation*}
Denote the variables on $\mathbb{G}$ by $x=(x',x'')\in \mathbb{G},$ where $x'$ corresponds to the first stratum.
For precise definitions we refer to Section \ref{SEC:prelim}.

Thus, to summarise briefly, in this paper we establish the following results:
\begin{itemize}
\item {\bf (Hardy inequalities)}
Let $\mathbb{G}$ be a stratified group
with $N$ being the dimension of the first stratum, and let $\alpha,\,\beta\in \mathbb{R}$.
Then for all complex-valued functions $f\in C^{\infty}_{0}(\mathbb{G}\backslash\{x'=0\})$ and $1<p<\infty,$ we have the following $L^{p}$-Caffarelli-Kohn-Nirenberg type inequality
\begin{equation}\label{introeqLpCKN}
\frac{|N-\gamma|}{p}
\left\|\frac{f}{|x'|^{\frac{\gamma}{p}}}\right\|^{p}_{L^{p}(\mathbb{G})}\leq
\left\|\frac{1}{|x'|^{\alpha}} \nabla_{H} f\right\|_{L^{p}(\mathbb{G})}\left\|\frac{f}{|x'|^{\frac{\beta}{p-1}}}\right\|^{p-1}_{L^{p}(\mathbb{G})},
\end{equation}
where $\gamma=\alpha+\beta+1$ and $|\cdot|$ is the Euclidean norm on $\mathbb{R}^{N}$. If $\gamma\neq N$ then the constant $\frac{|N-\gamma|}{p}$ is sharp.
In the special case of $\alpha=0$, $\beta=p-1$ and $\gamma=p$, inequality \eqref{introeqLpCKN}
implies
\begin{equation}\label{EQ:Hardy-mod}
\frac{|N-p|}{p}\left\|\frac{1}{|x'|}f
\right\|_{L^{p}(\mathbb{G})}\leq\left\|\nabla_{H} f\right\|_{L^{p}(\mathbb{G})},\quad 1<p<\infty,
\end{equation}
where the constant $\frac{|N-p|}{p}$ is sharp for $p\not=N$. One novelty of this is that we do not
require that $p<N$. In turn, for $1<p<N$,
the inequality \eqref{EQ:Hardy-mod} gives a stratified group version of $L^{p}$-Hardy inequality
\begin{equation}\label{EQ:Hardy-N}
\left\|\frac{1}{|x'|}f
\right\|_{L^{p}(\mathbb{G})}\leq\frac{p}{N-p}\left\|\nabla_{H} f\right\|_{L^{p}(\mathbb{G})},
\quad 1<p<N,
\end{equation}
again with $\frac{p}{N-p}$ being the best constant.
\item {\bf (Badiale-Tarantello conjecture)}
Let $x=(x',x'')\in\mathbb{R}^{N}\times\mathbb{R}^{n-N}.$ In \cite{BadTar:ARMA-2002}
Badiale and Tarantello proved that for $2\leq N\leq n$ and $1\leq p<N$ there exists a constant $C_{n,N,p}$ such that
\begin{equation}\label{BT}
\left\|\frac{1}{|x'|}f
\right\|_{L^{p}(\mathbb{R}^{n})}\leq C_{n,N,p}\left\|\nabla f\right\|_{L^{p}(\mathbb{R}^{n})},
\end{equation}
where $\nabla$ is the standard Euclidean gradient.
Clearly, for $N=n$ this gives the classical Hardy's inequality with
the best constant
$$C_{n,p}=\frac{p}{n-p}.$$
It was conjectured in \cite[Remark 2.3]{BadTar:ARMA-2002} that
the best constant in \eqref{BT} is given by
\begin{equation}
C_{N,p}=\frac{p}{N-p}.
\end{equation}
This conjecture was proved in \cite{Secchi-Smets-Willen}.
As a consequence of our techniques, we give a new proof of the Badiale-Tarantello conjecture.
\item {\bf (Critical Hardy inequality)}
For $p=N$, the inequality \eqref{EQ:Hardy-N} fails. In this case the Hardy inequality \eqref{HRn-p} is replaced by a logarithmic version, an analogue of which we establish on stratified groups as well.
For a bounded domain $\Omega\subset \mathbb{G}$ with $0\in\Omega$ and $f\in C^{\infty}_{0}(\Omega\backslash\{x'=0\})$ we have
\begin{equation}\label{introeqcritcase}
\left\|\frac{f}{|x'|{\rm log}\frac{R}{|x'|}}
\right\|_{L^{N}(\Omega)}\leq  \frac{N}{N-1}\left\|\frac{x'}{|x'|}\cdot\nabla_{H} f(x)\right\|_{L^{N}(\Omega)},\quad N\geq2,
\end{equation}
where $R=\underset{x\in\Omega}{\sup}|x'|.$
In the abelian case of $\mathbb{G}=\mathbb R^{n}$ being the Euclidean space, inequality \eqref{introeqcritcase} reduces to the logarithmic Hardy inequality of Edmunds and Triebel \cite{ET99}.

\item {\bf ($p$-sub-Laplacian)}
Let $\mathbb{G}$ be a stratified group
with $N$ being the dimension of the first stratum, and let $1<p<\infty$ with $\frac{1}{p}+\frac{1}{q}=1$ and $\alpha,\,\beta\in \mathbb{R}$ be such that $\frac{p-N}{p-1}\leq\gamma:=\alpha+\beta+1\leq 0$.
Then for all $f\in C^{\infty}_{0}(\mathbb{G}\backslash\{x'=0\})$ we have
\begin{equation}\label{eqLpHR}
\frac{N+\gamma(p-1)-p}{p}
\left\|\frac{\nabla_{H}f}{|x'|^{\frac{\gamma}{p}}}\right\|^{p}_{L^{p}(\mathbb{G})}\leq
\left\|\frac{1}{|x'|^{\alpha}} \mathcal{L}_{p} f\right\|_{L^{p}(\mathbb{G})}\left\|\frac{\nabla_{H} f}{|x'|^{\beta}}\right\|_{L^{q}(\mathbb{G})},
\end{equation}
where $|\cdot|$ is the Euclidean norm on $\mathbb{R}^{N}$ and $\mathcal{L}_{p}$ is the $p$-sub-Laplacian operator defined in \eqref{pLap}.

\item {\bf (Higher order Hardy-Rellich inequalities)}
Let $1<p<\infty.$
For any $k,m\in \mathbb{N}$ we have
\begin{align*}
\frac{|N-\gamma|}{p}
& \left\|\frac{f}{|x'|^{\frac{\gamma}{p}}}  \right\|^{p}_{L^{p}
(\mathbb{G})} \\ & \leq
\widetilde{A}_{\alpha,m}\widetilde{A}_{\beta,k}\left\|\frac{1}{|x'|^{\alpha-m}}\nabla_{H}^{m+1}
f\right\|_{L^{p}(\mathbb{G})}\left\|\frac{1}{|x'|^{\frac{\beta}{p-1}-k}}\nabla_{H}^{k}f\right\|^{p-1}_{L^{p}(\mathbb{G})},
\end{align*}
for any real-valued function $f\in C^{\infty}_{0}(\mathbb{G}\backslash\{x'=0\})$,  $\gamma=\alpha+\beta+1$, and $\alpha\in\mathbb{R}$ such that $\prod_{j=0}^{m-1}
\left|N-p(\alpha-j)\right|\neq0,$ and
$$\widetilde{A}_{\alpha,m}:=p^{m}\left[\prod_{j=0}^{m-1}
\left|N-p(\alpha-j)\right|\right]^{-1},$$
as well as $\beta\in\mathbb{R}$ such that $\prod_{j=0}^{k-1}
\left|N-p(\frac{\beta}{p-1}-j)\right|\neq0,$ and
$$\widetilde{A}_{\beta,k}:=p^{k(p-1)}\left[\prod_{j=0}^{k-1}
\left|N-p\left(\frac{\beta}{p-1}-j\right)\right|\right]^{-(p-1)}.$$

\item {\bf ($L^{N}$-Poincar\'e inequality)}
The following $L^{N}$-Poincar\'e inequality (see e.g. \cite{Tak15}) for the horizontal gradient is proved:
$$\|f\|_{L^{N}(\Omega)}\leq R\left\|\nabla_{H}f\right\|_{L^{N}(\Omega)},$$
 where $R=\underset{x\in\Omega}{\sup}|x'|,$ for $f\in C^{\infty}_{0}(\mathbb{G})$ and
 any bounded domain $\Omega\subset\mathbb{G}.$
Note that the inequality \eqref{EQ:Hardy-mod} implies
\begin{equation}\label{Lp-Poincare}
\frac{|N-p|}{Rp}\left\|f
\right\|_{L^{p}(\Omega)}\leq\left\|\nabla_{H} f\right\|_{L^{p}(\Omega)},\quad 1<p<\infty,
\end{equation}
for $f\in C^{\infty}_{0}(\Omega\backslash\{x'=0\})$ and $R=\underset{x\in\Omega}{\sup}|x'|$. However, \eqref{Lp-Poincare}
gives a trivial inequality when $N=p$.

\item {\bf (Weighted $p$-sub-Laplacian)}
Let $0\leq F\in C^{\infty}(\mathbb{G})$ and $0\leq\eta\in L^{1}_{loc}(\mathbb{G})$ be such that
\begin{equation}\label{weighteddiv}
\eta F^{p-1}\leq -\mathcal{L}_{p,\rho}F,\quad {\rm a.e.\; in}\;\mathbb{G},
\end{equation}
where $\mathcal{L}_{p,\rho}$ is a weighted $p$-sub-Laplacian defined in \eqref{wsubL}.
Then we have
\begin{equation}\label{weightedthm1ineqn0}
\|\eta^{\frac{1}{p}}f\|^{p}_{L^{p}(\mathbb{G})}+C_{p}\left\|
\rho^{\frac{1}{p}}F\nabla_{H}\frac{f}{F}\right\|^{p}_{L^{p}(\mathbb{G})}
\leq\|\rho^{\frac{1}{p}}\nabla_{H}f\|^{p}_{L^{p}(\mathbb{G})},
\end{equation}
for all real-valued functions $f\in C^{\infty}_{0}(\mathbb{G})$ and $2\leq p<\infty$.
Here $C_{p}$ is a positive constant.
For $1<p<2$ the inequality \eqref{weightedthm1ineqn0} is replaced by an analogous one
while for $p=2$ it becomes an identity, see Remark \ref{remw2} and Remark \ref{remw1}, respectively.
\end{itemize}

In Section \ref{SEC:prelim} we very briefly recall the main concepts of stratified groups and fix the notation.
 In Section \ref{Sec3} we derive versions of $L^{p}$-Caffarelli-Kohn-Nirenberg type inequalities  on stratified groups and discuss their consequences including higher order cases as well as a new proof of the Badiale-Tarantello conjecture. An analogue of the critical Hardy inequality is proved in Section \ref{Sec4}.
Hardy-Rellich type inequalities and their weighted versions on stratified groups are presented and analysed in Section \ref{Sec5}.

\medskip

\section{Preliminaries}
\label{SEC:prelim}

A Lie group $\mathbb{G}=(\mathbb{R}^{n},\circ)$ is called a stratified group
(or a homogeneous Carnot group) if it satisfies the following conditions:

(a) For some natural numbers $N+N_{2}+...+N_{r}=n$, that is $N=N_{1},$
the decomposition $\mathbb{R}^{n}=\mathbb{R}^{N}\times...\times\mathbb{R}^{N_{r}}$ is valid, and
for every $\lambda>0$ the dilation $\delta_{\lambda}: \mathbb{R}^{n}\rightarrow \mathbb{R}^{n}$
given by
$$\delta_{\lambda}(x)\equiv\delta_{\lambda}(x',x^{(2)},...,x^{(r)}):=(\lambda x',\lambda ^{2}x^{(2)},...,\lambda^{r}x^{(r)})$$
is an automorphism of the group $\mathbb{G}.$ Here $x'\equiv x^{(1)}\in \mathbb{R}^{N}$ and $x^{(k)}\in \mathbb{R}^{N_{k}}$ for $k=2,...,r.$

(b) Let $N$ be as in (a) and let $X_{1},...,X_{N}$ be the left invariant vector fields on $\mathbb{G}$ such that
$X_{k}(0)=\frac{\partial}{\partial x_{k}}|_{0}$ for $k=1,...,N.$ Then
$${\rm rank}({\rm Lie}\{X_{1},...,X_{N}\})=n,$$
for every $x\in\mathbb{R}^{n},$ i.e. the iterated commutators
of $X_{1},...,X_{N}$ span the Lie algebra of $\mathbb{G}.$

That is, we say that the triple $\mathbb{G}=(\mathbb{R}^{n},\circ, \delta_{\lambda})$ is a stratified group. See also e.g. \cite{FR} for discussions from the Lie algebra point of view.
Here $r$ is called a step of $\mathbb{G}$ and the left invariant vector fields $X_{1},...,X_{N}$ are called
the (Jacobian) generators of $\mathbb{G}$.
The number
$$Q=\sum_{k=1}^{r}kN_{k},\quad N_{1}=N,$$
is called the homogeneous dimension of $\mathbb{G}$.
\smallskip
The second order differential operator
\begin{equation}\label{sublap}
\mathcal{L}=\sum_{k=1}^{N}X_{k}^{2},
\end{equation}
is called the (canonical) sub-Laplacian on $\mathbb{G}$.
The sub-Laplacian $\mathcal{L}$ is a left invariant homogeneous hypoelliptic differential operator and it is known that $\mathcal{L}$ is elliptic if and only if the step of $\mathbb{G}$ is equal to 1. We also recall that the standard Lebesque measure $dx$ on $\mathbb R^{n}$ is the Haar measure for $\mathbb{G}$ (see, e.g. \cite[Proposition 1.6.6]{FR}).
The left invariant vector field $X_{j}$ has an explicit form and satisfies the divergence theorem,
see e.g. \cite{Ruzhansky-Suragan:Layers} for the derivation of the exact formula: more precisely, we can write
\begin{equation}\label{Xk0}
X_{k}=\frac{\partial}{\partial x'_{k}}+
\sum_{l=2}^{r}\sum_{m=1}^{N_{l}}a_{k,m}^{(l)}(x',...,x^{(l-1)})
\frac{\partial}{\partial x_{m}^{(l)}},
\end{equation}
see also \cite[Section 3.1.5]{FR} for a general presentation.
We will also use the following notations
$$\nabla_{H}:=(X_{1},\ldots, X_{N})$$
for the horizontal gradient,
$${\rm div}_{H} v:=\nabla_{H}\cdot v$$
for the horizontal divergence,
\begin{equation}\label{pLap}
 \mathcal{L}_{p}f:={\rm div}_{H}(|\nabla_{H}f|^{p-2}\nabla_{H}f),\quad 1<p<\infty,
\end{equation}
for the horizontal $p$-Laplacian (or $p$-sub-Laplacian), and

$$|x'|=\sqrt{x'^{2}_{1}+\ldots+x'^{2}_{N}}$$ for the Euclidean norm on $\mathbb{R}^{N}.$

The explicit representation \eqref{Xk0} allows us to have the identities
\begin{equation}\label{gradgamma}
|\nabla_{H}|x'|^{\gamma}|=\gamma|x'|^{\gamma-1},
\end{equation}
and
\begin{equation}\label{divgamma}
{\rm div}_{H}\left(\frac{x'}{|x'|^{\gamma}}\right)=\frac{\sum_{j=1}^{N}|x'|^{\gamma}X_{j}x'_{j}-\sum_{j=1}^{N}x'_{j}\gamma|x'|^{\gamma-1}X_{j}|x'| }{|x'|^{2\gamma}}=\frac{N-\gamma}{|x'|^{\gamma}}
\end{equation}
 for all $\gamma\in\mathbb{R},\; |x'|\neq 0.$

\section{Horizontal $L^{p}$-Caffarelli-Kohn-Nirenberg type inequalities and consequences}
\label{Sec3}In this section and in the sequel we adopt all the notation introduced in Section \ref{SEC:prelim} concerning stratified groups and the horizontal operators.
\subsection{Caffarelli-Kohn-Nirenberg inequalities}
In this section we establish the following horizontal $L^{p}$-Caffarelli-Kohn-Nirenberg type inequalities  on the stratified group $\mathbb{G}$ and then discuss their consequences and proofs. The proof is analogous to \cite{ORS16} in the case of homogeneous groups, but here we rely on the
divergence theorem rather on the polar decomposition which is less suitable for the stratified setting.
We refer e.g. to \cite{CKN-1984} and \cite{CW-2001} for Euclidean settings of Caffarelli-Kohn-Nirenberg inequalities.

\begin{thm}\label{LpCKN}
Let $\mathbb{G}$ be a homogeneous stratified group
with $N$ being the dimension of the first stratum, and let $\alpha,\,\beta\in \mathbb{R}$.
Then for any $f\in C^{\infty}_{0}(\mathbb{G}\backslash\{x'=0\}),$ and all $1<p<\infty,$ we have
\begin{equation}\label{eqLpCKN}
\frac{|N-\gamma|}{p}
\left\|\frac{f}{|x'|^{\frac{\gamma}{p}}}\right\|^{p}_{L^{p}(\mathbb{G})}\leq
\left\|\frac{1}{|x'|^{\alpha}} \nabla_{H} f\right\|_{L^{p}(\mathbb{G})}\left\|\frac{f}{|x'|^{\frac{\beta}{p-1}}}\right\|^{p-1}_{L^{p}(\mathbb{G})},
\end{equation}
where $\gamma=\alpha+\beta+1$ and $|\cdot|$ is the Euclidean norm on $\mathbb{R}^{N}$. If $\gamma\neq N$ then the constant $\frac{|N-\gamma|}{p}$ is sharp.
\end{thm}

In the abelian case ${\mathbb G}=(\mathbb R^{n},+)$, we have
$N=n$, $\nabla_{H}=\nabla=(\partial_{x_{1}},\ldots,\partial_{x_{n}})$, so \eqref{eqLpCKN}  implies the $L^{p}$-Caffarelli-Kohn-Nirenberg type inequality (see e.g. \cite{Costa-2008} and \cite{DJSJ-2013}) for $\mathbb{G}\equiv\mathbb{R}^{n}$ with the sharp constant:
\begin{equation}\label{CKN}
\frac{|n-\gamma|}{p}
\left\|\frac{f}{\|x\|^{\frac{\gamma}{p}}}\right\|^{p}_{L^{p}(\mathbb{R}^{n})}\leq
\left\|\frac{1}{\|x\|^{\alpha}}\nabla f\right\|_{L^{p}(\mathbb{R}^{n})}\left\|\frac{f}{\|x\|^{\frac{\beta}{p-1}}}\right\|^{p-1}_{L^{p}(\mathbb{R}^{n})},
\end{equation}
for all $f\in C_{0}^{\infty}(\mathbb{R}^{n}\backslash\{0\}),$ and $\|x\|=\sqrt{x_{1}^{2}+\ldots+x_{n}^{2}}.$
In the case
$$\beta=\gamma\left(1-\frac{1}{p}\right),$$
that is, taking $\beta=(\alpha+1)(p-1)$ and $\gamma=p(\alpha+1)$, the inequality
\eqref{eqLpCKN} implies that
\begin{equation}\label{Lpweighted}
\frac{|N-p(\alpha+1)|}{p}\left\|\frac{f}{|x'|^{\alpha+1}}\right\|_{L^{p}(\mathbb{G})}\leq
\left\|\frac{1}{|x'|^{\alpha}}\nabla_{H} f\right\|_{L^{p}(\mathbb{G})},\quad 1<p<\infty,
\end{equation}
for any $f\in C^{\infty}_{0}(\mathbb{G}\backslash\{x'=0\})$ and all $\alpha\in\mathbb{R}$.

When $\alpha=0$ and $1<p<N$, the inequality \eqref{Lpweighted} gives the following stratified group version of $L^{p}$-Hardy inequality
\begin{equation}\label{47-1}
\left\|\frac{1}{|x'|}f
\right\|_{L^{p}(\mathbb{G})}\leq
\frac{p}{N-p}\left\|\nabla_{H} f\right\|_{L^{p}(\mathbb{G})},\; 1<p<N,
\end{equation}
again with $\frac{p}{N-p}$ being the best constant
(see \cite{DAmbrosio-Difur04} and \cite{Yen16} for the version on the Heisenberg group).
In the abelian case ${\mathbb G}=(\mathbb R^{n},+)$, $n\geq 3$, \eqref{47-1} implies the classical Hardy inequality for $\mathbb{G}\equiv\mathbb{R}^{n}$:
\begin{equation*}\label{Hardy}
\left\|\frac{f}{\|x\|}\right\|_{L^{p}(\mathbb{R}^{n})}\leq
\frac{p}{n-p}\left\|\nabla f\right\|_{L^{p}(\mathbb{R}^{n})},
\end{equation*}
for all $f\in C_{0}^{\infty}(\mathbb{R}^{n}\backslash\{0\}),$ and $\|x\|=\sqrt{x_{1}^{2}+\ldots+x_{n}^{2}}.$

The inequality \eqref{47-1} implies the following Heisenberg-Pauli-Weyl type  uncertainly principle on stratified groups (see e.g. \cite{Ciatti-Cowling-Ricci}, \cite{Ruzhansky-Suragan:Hardy}, \cite{Ruzhansky-Suragan:Layers} and \cite{ORS16} for different settings):
For each $f\in C^{\infty}_{0}(\mathbb{G}\backslash\{x'=0\})$, using H\"older's inequality and \eqref{47-1}, we have
\begin{multline}
\left\|f\right\|^{2}_{L^{2}(\mathbb{G})}
\leq \left\|\frac{1}{|x'|} f\right\|_{L^{p}(\mathbb{G})}\left\||x'| f\right\|_{L^{\frac{p}{p-1}}(\mathbb{G})}
\\\leq\frac{p}{N-p}\left\|\nabla_{H} f\right\|_{L^{p}(\mathbb{G})}\left\||x'| f\right\|_{L^{\frac{p}{p-1}}(\mathbb{G})},\quad 1<p<N,
\end{multline}
that is,
\begin{equation}\label{1UP1p}
\left\|f\right\|^{2}_{L^{2}(\mathbb{G})}
\leq\frac{p}{N-p}\left\|\nabla_{H} f\right\|_{L^{p}(\mathbb{G})}\left\||x'| f\right\|_{L^{\frac{p}{p-1}}(\mathbb{G})},\quad 1<p<N.
\end{equation}

In the abelian case ${\mathbb G}=(\mathbb R^{n},+)$, taking
$N=n$, we obtain that \eqref{1UP1p} with $p=2$ implies
the classical
uncertainty principle for $\mathbb{G}\equiv\mathbb R^{n}$: for all $f\in C_{0}^{\infty}(\mathbb{R}^{n}\backslash \{0\}),$ we have
\begin{equation*}\label{UPRn}
\left(\int_{\mathbb R^{n}}
 |f(x)|^{2} dx\right)^{2}\leq\left(\frac{2}{n-2}\right)^{2}\int_{\mathbb R^{n}}|\nabla f(x)|^{2}dx
\int_{\mathbb R^{n}} \|x\|^{2} |f(x)|^{2}dx,
\end{equation*}
which is the Heisenberg-Pauli-Weyl uncertainly principle on $\mathbb R^{n}$.

On the other hand, directly from the inequality \eqref{eqLpCKN}, using the H\"older inequality, we can obtain
a number of Heisenberg-Pauli-Weyl type uncertainly inequalities which have
various consequences and applications.
For instance, when $\alpha p=\alpha+\beta+1$, we get
\begin{equation}\label{HPW1}
\frac{|N-\alpha p|}{p}\left\|\frac{f}{|x'|^{\alpha}}\right\|^{p}_{L^{p}(\mathbb{G})}
\leq\left\|\frac{\nabla_{H} f}{|x'|^{\alpha}}\right\|_{L^{p}(\mathbb{G})}\left\||x'|^{\frac{1}{p-1}-\alpha} f\right\|^{p-1}_{L^{p}(\mathbb{G})},
\end{equation}
and if $0=\alpha+\beta+1$ and $\alpha=-p$, then
\begin{equation}\label{HPW2}
\frac{N}{p}\left\|f\right\|^{p}_{L^{p}(\mathbb{G})}
\leq\left\||x'|^{p}\nabla_{H} f\right\|_{L^{p}(\mathbb{G})}\left\| \frac{f}{|x'|}\right\|^{p-1}_{L^{p}(\mathbb{G})},
\end{equation}
both with sharp constants.

\begin{proof}[Proof of Theorem \ref{LpCKN}]
We may assume that $\gamma\neq N$ since for $\gamma=N$ the inequality \eqref{eqLpCKN} is trivial. By using the identity \eqref{divgamma}, the divergence theorem and Schwarz's inequality one calculates
\begin{align*}
\int_{\mathbb{G}}
\frac{|f(x)|^{p}}
{|x'|^{\gamma}}dx
& =\frac{1}{N-\gamma}\int_{\mathbb{G}}|f(x)|^{p} {\rm div}_{H}\left(\frac{x'}{|x'|^{\gamma}}\right)dx
\\ & =-\frac{1}{N-\gamma}{\rm Re}\int_{\mathbb{G}} pf(x)|f(x)|^{p-2}\frac{\overline{x'\cdot\nabla_{H}f}}{|x'|^{\gamma}}dx
\\ &
\leq\left|\frac{p}{N-\gamma}\right| \int_{\mathbb{G}}
\frac{|f(x)|^{p-1}}{|x'|^{\gamma}} \left| x'\cdot\nabla_{H}f\right| dx
\\ & \leq\left|\frac{p}{N-\gamma}\right| \int_{\mathbb{G}}
\frac{|f(x)|^{p-1}}{|x'|^{\alpha+\beta}}
\left|\nabla_{H}f(x)\right|dx
\\ & \leq \left|\frac{p}{N-\gamma}\right| \left(\int_{\mathbb{G}}
\frac{\left|\nabla_{H}f(x)\right|^{p}}{|x'|^{\alpha p}}
dx\right)^{\frac{1}{p}} \left(\int_{\mathbb{G}}
\frac{|f(x)|^{p}}{|x'|^{\frac{\beta p}{p-1}}}
dx\right)^{\frac{p-1}{p}}.
\end{align*}
Here we have used H\"older's inequality in the last line. Thus, we arrive at
\begin{equation}\label{calcul}
\left|\frac{N-\gamma}{p}\right| \int_{\mathbb{G}}
\frac{|f(x)|^{p}}
{|x'|^{\gamma}}dx \leq \left(\int_{\mathbb{G}}
\frac{\left|\nabla_{H}f(x)\right|^{p}}{|x'|^{\alpha p}}
dx\right)^{\frac{1}{p}} \left(\int_{\mathbb{G}}
\frac{|f(x)|^{p}}{|x'|^{\frac{\beta p}{p-1}}}
dx\right)^{\frac{p-1}{p}}.
\end{equation}
This proves \eqref{eqLpCKN}.
Now it remains to show the sharpness of the constant. Let us examine the equality
condition in above H\"older's inequality as in the abelian case (see \cite{DJSJ-2013}).
For this we consider the function
\begin{equation}
g(x)=\left\{
\begin{array}{ll}
    e^{-\frac{C}{\lambda}|x'|^{\lambda}},\quad \lambda:=\alpha-\frac{\beta}{p-1}+1\neq0,\\
    \frac{1}{|x'|^{C}},\quad \alpha-\frac{\beta}{p-1}+1=0, \\
\end{array}
\right.
\label{EQ:fs}
\end{equation}
where $C=\left|\frac{N-\gamma}{p}\right|$ and $\gamma\neq N.$
Then it can be checked that
\begin{equation}
\left|\frac{p}{N-\gamma}\right|^{p}
\frac{\left|\nabla_{H}g(x)\right|^{p}}{|x'|^{\alpha p}}=
\frac{|g(x)|^{p}}{|x'|^{\frac{\beta p}{p-1}}},
\end{equation}
which satisfies the equality
condition in H\"older's inequality.
It shows that the constant $C=\left|\frac{N-\gamma}{p}\right|$ is sharp.
\end{proof}

\subsection{Badiale-Tarantello conjecture.} The proof of Theorem \ref{LpCKN} gives the following
similar statement in $\mathbb{R}^{n}.$
\begin{prop}\label{LpBT}
Let $x=(x',x'')\in\mathbb{R}^{N}\times\mathbb{R}^{n-N},\; 1\leq N\leq n,$ and $\alpha,\,\beta\in \mathbb{R}$.
Then for any $f\in C^{\infty}_{0}(\mathbb{R}^{n}\backslash\{x'=0\}),$ and all $1<p<\infty,$ we have
\begin{equation}\label{eqBT}
\frac{|N-\gamma|}{p}
\left\|\frac{f}{|x'|^{\frac{\gamma}{p}}}\right\|^{p}_{L^{p}(\mathbb{R}^{n})}\leq
\left\|\frac{1}{|x'|^{\alpha}} \nabla f\right\|_{L^{p}(\mathbb{R}^{n})}\left\|\frac{f}{|x'|^{\frac{\beta}{p-1}}}\right\|^{p-1}_{L^{p}(\mathbb{R}^{n})},
\end{equation}
where $\gamma=\alpha+\beta+1$ and $|x'|$ is the Euclidean norm on $\mathbb{R}^{N}$. If $\gamma\neq N$ then the constant $\frac{|N-\gamma|}{p}$ is sharp.
\end{prop}

The proof is similar to the proof of Theorem \ref{LpCKN}. However, for the sake of completeness here we give the details.

\begin{proof}[Proof of Proposition \ref{LpBT}]
We may assume that $\gamma\neq N$ since for $\gamma=N$ the inequality \eqref{eqBT} is trivial. By using the identity
$${\rm div}_{N}\frac{x'}{|x'|^{\gamma}}=\frac{N-\gamma}{|x'|^{\gamma}},$$
for all $\gamma\in \mathbb{R}$ and $x'\in\mathbb{R}^{N}$ with $|x'|\neq 0$, where
 ${\rm div}_{N}$ is
the standard divergence on $\mathbb{R}^{N}$, the divergence theorem and Schwarz's inequality one calculates
\begin{align*}
\int_{\mathbb{R}^{n}}
\frac{|f(x)|^{p}}
{|x'|^{\gamma}}dx
& =\frac{1}{N-\gamma}\int_{\mathbb{R}^{n}}|f(x)|^{p} {\rm div}_{N}\left(\frac{x'}{|x'|^{\gamma}}\right)dx
\\ & =-\frac{1}{N-\gamma}{\rm Re}\int_{\mathbb{R}^{n}} pf(x)|f(x)|^{p-2}\frac{\overline{x'\cdot\nabla_{N}f}}{|x'|^{\gamma}}dx
\\ &
\leq\left|\frac{p}{N-\gamma}\right| \int_{\mathbb{R}^{n}}
\frac{|f(x)|^{p-1}}{|x'|^{\gamma}} \left| x'\cdot\nabla_{N}f\right| dx
\\ &
=\left|\frac{p}{N-\gamma}\right| \int_{\mathbb{R}^{n}}
\frac{|f(x)|^{p-1}}{|x'|^{\gamma}} \left| x'_{0}\cdot\nabla f\right| dx
\\ &
\leq\left|\frac{p}{N-\gamma}\right| \int_{\mathbb{R}^{n}}
\frac{|f(x)|^{p-1}}{|x'|^{\alpha+\beta}}
\left|\nabla f(x)\right|dx
\\ &
\leq \left|\frac{p}{N-\gamma}\right| \left(\int_{\mathbb{R}^{n}}
\frac{\left|\nabla f(x)\right|^{p}}{|x'|^{\alpha p}}
dx\right)^{\frac{1}{p}} \left(\int_{\mathbb{R}^{n}}
\frac{|f(x)|^{p}}{|x'|^{\frac{\beta p}{p-1}}}
dx\right)^{\frac{p-1}{p}},
\end{align*}
where $x'_{0}=(x',0)\in \mathbb{R}^{n}$, that is $|x'_{0}|=|x'|,$ $\nabla_{N}$ is the standard gradient on $\mathbb{R}^{N}$, as well as $\nabla$ is the gradient
on $\mathbb{R}^{n}.$
Here we have used H\"older's inequality in the last line. Thus, we arrive at
\begin{equation}\label{calcul}
\left|\frac{N-\gamma}{p}\right| \int_{\mathbb{R}^{n}}
\frac{|f(x)|^{p}}
{|x'|^{\gamma}}dx \leq \left(\int_{\mathbb{R}^{n}}
\frac{\left|\nabla f(x)\right|^{p}}{|x'|^{\alpha p}}
dx\right)^{\frac{1}{p}} \left(\int_{\mathbb{R}^{n}}
\frac{|f(x)|^{p}}{|x'|^{\frac{\beta p}{p-1}}}
dx\right)^{\frac{p-1}{p}}.
\end{equation}
This proves \eqref{eqBT}.
Now it remains to show the sharpness of the constant. Let us examine the equality
condition in above H\"older's inequality.
Consider
\begin{equation}
g(x)=\left\{
\begin{array}{ll}
    e^{-\frac{C}{\lambda}|x'|^{\lambda}},\quad \lambda:=\alpha-\frac{\beta}{p-1}+1\neq0,\\
    \frac{1}{|x'|^{C}},\quad \alpha-\frac{\beta}{p-1}+1=0, \\
\end{array}
\right.
\label{EQ:fs}
\end{equation}
where $C=\left|\frac{N-\gamma}{p}\right|$ and $\gamma\neq N.$
Then it can be checked that
\begin{equation}
\left|\frac{p}{N-\gamma}\right|^{p}
\frac{\left|\nabla g(x)\right|^{p}}{|x'|^{\alpha p}}=\left|\frac{p}{N-\gamma}\right|^{p}
\frac{\left|\nabla_{N}g(x)\right|^{p}}{|x'|^{\alpha p}}=
\frac{|g(x)|^{p}}{|x'|^{\frac{\beta p}{p-1}}},
\end{equation}
which satisfies the equality
condition in H\"older's inequality.
It shows that the constant $C=\left|\frac{N-\gamma}{p}\right|$ is sharp.
\end{proof}

As above, taking $\beta=(\alpha+1)(p-1)$ and $\gamma=p(\alpha+1)$ the inequality
\eqref{eqBT} implies that
\begin{equation}\label{LpweightedBT}
\frac{|N-p(\alpha+1)|}{p}\left\|\frac{f}{|x'|^{\alpha+1}}\right\|_{L^{p}(\mathbb{R}^{n})}\leq
\left\|\frac{1}{|x'|^{\alpha}}\nabla f\right\|_{L^{p}(\mathbb{R}^{n})},\; 1<p<\infty,
\end{equation}
for any $f\in C^{\infty}_{0}(\mathbb{R}^{n}\backslash\{x'=0\})$ and for all $\alpha\in\mathbb{R}$ with the sharp constant.
When $\alpha=0$ and $1<p<N,\; 2\leq N\leq n,$ the inequality \eqref{LpweightedBT} implies that
\begin{equation}
\left\|\frac{1}{|x'|}f
\right\|_{L^{p}(\mathbb{R}^{n})}\leq
\frac{p}{N-p}\left\|\nabla f\right\|_{L^{p}(\mathbb{R}^{n})},
\end{equation}
again with $\frac{p}{N-p}$ being the best constant. This proves the Badiale-Tarantello conjecture, which is stated in the introduction (see also \cite[Remark 2.3]{BadTar:ARMA-2002} for the original statement).

\subsection{Horizontal higher order versions}
In this subsection we show how by iterating the established $L^{p}$-Caffarelli-Kohn-Nirenberg type inequalities one can get inequalities of higher order.
Putting $|\nabla_{H} f|$ instead of $f$ and $\alpha-1$
instead of $\alpha$ in \eqref{Lpweighted} we consequently have

$$
\left\|\frac{\nabla_{H} f}{|x'|^{\alpha}}\right\|_{L^{p}(\mathbb{G})}\leq \frac{p}{|N-p\alpha|}
\left\|\frac{1}{|x'|^{\alpha-1}}\nabla_{H}^{2} f\right\|_{L^{p}(\mathbb{G})},
$$
for $\alpha\neq \frac{N}{p}.$ Here and after we understand $\nabla_{H}^{2} f=\nabla_{H}|\nabla_{H} f|$, that is, $\nabla_{H}^{m} f=\nabla_{H}|\nabla^{m-1}_{H} f|,\;m\in\mathbb{N}.$
Combining it with \eqref{Lpweighted} we get

\begin{equation}\label{iter1}
\left\|\frac{f}{|x'|^{\alpha+1}}\right\|_{L^{p}(\mathbb{G})}\leq \frac{p}{|N-p(\alpha+1)|}\frac{p}{|N-p\alpha)|}
\left\|\frac{1}{|x'|^{\alpha-1}}\nabla_{H}^{2} f\right\|_{L^{p}(\mathbb{G})},
\end{equation}
for each $\alpha\in\mathbb{R}$ such that $\alpha\neq \frac{N}{p}-1$ and $\alpha\neq \frac{N}{p}.$
This iteration process gives

\begin{equation}\label{Lph1}
\left\|\frac{f}{|x'|^{\theta+1}}\right\|_{L^{p}(\mathbb{G})}\leq A_{\theta,k}
\left\|\frac{1}{|x'|^{\theta+1-k}}\nabla_{H}^{k} f\right\|_{L^{p}(\mathbb{G})},\; 1<p<\infty,
\end{equation}
for any $f\in C^{\infty}_{0}(\mathbb{G}\backslash\{x'=0\})$ and all $\theta\in\mathbb{R}$ such that $\prod_{j=0}^{k-1}
\left|N-p(\theta+1-j)\right|\neq0,$ and
$$A_{\theta,k}:=p^{k}\left[\prod_{j=0}^{k-1}
\left|N-p(\theta+1-j)\right|\right]^{-1}.$$
Similarly, we have
\begin{equation}\label{Lph2}
\left\|\frac{\nabla_{H}f}{|x'|^{\vartheta+1}}\right\|_{L^{p}(\mathbb{G})}\leq A_{\vartheta,m}
\left\|\frac{1}{|x'|^{\vartheta+1-m}}\nabla_{H}^{m+1} f\right\|_{L^{p}(\mathbb{G})},\; 1<p<\infty,
\end{equation}
for any $f\in C^{\infty}_{0}(\mathbb{G}\backslash\{x'=0\})$ and all $\vartheta\in\mathbb{R}$ such that $\prod_{j=0}^{m-1}
\left|N-p(\vartheta+1-j)\right|\neq0,$ and
$$A_{\vartheta,m}:=p^{m}\left[\prod_{j=0}^{m-1}
\left|N-p(\vartheta+1-j)\right|\right]^{-1}.$$
Now putting $\vartheta+1=\alpha$ and $\theta+1=\frac{\beta}{p-1}$  into \eqref{Lph2} and \eqref{Lph1}, respectively, from \eqref{eqLpCKN} we obtain

\begin{prop}
Let $1<p<\infty.$
For any $k,m\in \mathbb{N}$ we have
\begin{equation}\label{Lphighorder}
\frac{|N-\gamma|}{p}
\left\|\frac{f}{|x'|^{\frac{\gamma}{p}}}\right\|^{p}_{L^{p}
(\mathbb{G})}\leq
\widetilde{A}_{\alpha,m}\widetilde{A}_{\beta,k}\left\|\frac{1}{|x'|^{\alpha-m}}\nabla_{H}^{m+1}
f\right\|_{L^{p}(\mathbb{G})}\left\|\frac{1}{|x'|^{\frac{\beta}{p-1}-k}}\nabla_{H}^{k}f\right\|^{p-1}_{L^{p}(\mathbb{G})},
\end{equation}
for any $f\in C^{\infty}_{0}(\mathbb{G}\backslash\{x'=0\})$,  $\gamma=\alpha+\beta+1$, and $\alpha\in\mathbb{R}$ such that $\prod_{j=0}^{m-1}
\left|N-p(\alpha-j)\right|\neq0,$ and
$$\widetilde{A}_{\alpha,m}:=p^{m}\left[\prod_{j=0}^{m-1}
\left|N-p(\alpha-j)\right|\right]^{-1},$$
as well as $\beta\in\mathbb{R}$ such that $\prod_{j=0}^{k-1}
\left|N-p(\frac{\beta}{p-1}-j)\right|\neq0,$ and
$$\widetilde{A}_{\beta,k}:=p^{k(p-1)}\left[\prod_{j=0}^{k-1}
\left|N-p\left(\frac{\beta}{p-1}-j\right)\right|\right]^{-(p-1)}.$$
\end{prop}

\section{Horizontal critical Hardy type inequality}
\label{Sec4}

For $p=N$ the inequality \eqref{47-1} fails for any constant (see, e.g., \cite{ET99} and \cite{IIO} for discussions in Euclidean cases). However, we state the following theorem for the (critical) case $p=N$.

\begin{thm} \label{critcase}
For a bounded domain $\Omega\subset \mathbb{G}$ with $0\in\Omega$ and all $f\in C^{\infty}_{0}(\Omega\backslash\{x'=0\})$ we have
\begin{equation}\label{eqcritcase}
\left\|\frac{f(x)}{|x'|{\rm log}\frac{R}{|x'|}}
\right\|_{L^{N}(\Omega)}\leq  \frac{N}{N-1}\left\|\frac{x'}{|x'|}\cdot\nabla_{H} f(x)\right\|_{L^{N}(\Omega)},\; 1<N<\infty,
\end{equation}
where $R=\underset{x\in\Omega}{\sup}|x'|.$
\end{thm}

Note that below we give the proof of \eqref{eqcritcase} for real-valued functions, the same inequality follows for all complex-valued functions by using the identity (cf. Davies \cite[p. 176]{Davies-bk:Semigroups-1980})
\begin{equation}\label{EQ:Davies-rc}
\forall z\in\mathbb C:\;
|z|^{p}=\left(\int_{-\pi}^{\pi}|\cos\theta|^{p} d\theta\right)^{-1}
\int_{-\pi}^{\pi}\left| {\rm Re}(z)\cos\theta+{\rm Im}(z)\sin\theta\right|^{p}d\theta,
\end{equation}
which follows from the representation $z=r(\cos\phi+i\sin\phi)$ by some manipulations.
To prove \eqref{critcase} we follow the Euclidean setting from \cite{Tak15}. First let us prove the following more abstract theorem, and then the proof of Theorem \ref{critcase} will follow easily from this.
\begin{thm} \label{critcaseMAIN}
Let $0\in\Omega\subset \mathbb{G}$ be a bounded domain.
Let $g:(1,\infty)\rightarrow \mathbb{R}$ be a $C^{2}$-function such that
\begin{equation}\label{lem1}
g^{\prime}(t)<0,\quad g^{\prime\prime}(t)>0
\end{equation}
for all $t>1$ and
\begin{equation}\label{lem2}
\frac{(-g^{\prime}(t))^{2(N-1)}}{(g^{\prime\prime}(t))^{N-1}}\leq C<\infty,\quad \forall t>1.
\end{equation}
Then we have
\begin{multline}\label{lem3}
\left(\frac{N-1}{N}\right)^{N}\int_{\Omega}\frac{|f(x)|^{N}}{|x'|^{N}}
\left(-g^{\prime}\left( {\rm log}\frac{Re}{|x'|}\right)\right)^{N-2}g^{\prime\prime}\left( {\rm log}\frac{Re}{|x'|}\right)dx
\\
\leq \int_{\Omega} \frac{\left(-g^{\prime}\left({\rm log}\frac{Re}{|x'|}\right)\right)^{2(N-1)}}{\left(g^{\prime\prime}\left({\rm log}\frac{Re}{|x'|}\right)\right)^{N-1}}\left|\frac{x'}{|x'|}\cdot\nabla_{H} f(x)\right|^{N}dx,
\end{multline}
for all $f\in C^{\infty}_{0}(\Omega\backslash\{x'=0\}).$
Here $R=\underset{x\in\Omega}{\sup}|x'|.$
\end{thm}
\begin{proof}[Proof of Theorem \ref{critcase}]
If we take
$$g(t)=-{\rm log}(t-1),$$
for $t>1,$ then we see that this function satisfies all assumptions of Theorem \ref{critcaseMAIN}.
That is,

$$g^{\prime}(t)=-\frac{1}{t-1}<0,\quad g^{\prime\prime}(t)=\frac{1}{(t-1)^{2}}>0,$$
and
$$\frac{(-g^{\prime}(t))^{2(N-1)}}{(g^{\prime\prime}(t))^{N-1}}=1,\quad \forall t>1.$$
Therefore, putting
$$g^{\prime}\left({\rm log}\frac{Re}{|x'|}\right)=-\frac{1}{{\rm log}\frac{R}{|x'|}}$$
and
$$ g^{\prime\prime}\left({\rm log}\frac{Re}{|x'|}\right)=\frac{1}{\left({\rm log}\frac{R}{|x'|}\right)^{2}}$$
in \eqref{lem3} we obtain \eqref{eqcritcase}.
\end{proof}

\begin{rem}
Taking $$g(t)=e^{\frac{Nt}{1-N}},\;t>1,$$ in \eqref{lem3} we obtain
$$\|f\|_{L^{N}(\Omega)}\leq\left\||x'|
\frac{x'}{|x'|}\cdot\nabla_{H}f\right\|_{L^{N}(\Omega)}.$$
Since $R=\underset{x\in\Omega}{\sup}|x'|$ using Schwarz's inequality
we get
$$\|f\|_{L^{N}(\Omega)}\leq R\left\|\nabla_{H}f\right\|_{L^{N}(\Omega)},$$
which is $L^{N}$-Poincare inequality for the horizontal gradient.
\end{rem}

\begin{proof}[Proof of Theorem \ref{critcaseMAIN}]
Let us introduce notations
$$R_{\epsilon}:=\underset{x\in\Omega}{\sup}\sqrt{|x'|^{2}+2\epsilon^{2}},$$$$ F_{\epsilon}(x):={\rm log}\frac{R_{\epsilon}e}{\sqrt{|x'|^{2}+\epsilon^{2}}},$$
and
$$G_{\epsilon}(x)=g(F_{\epsilon}(x)),$$
for, say, $\epsilon>0.$
Then a direct calculation shows
$$|\nabla_{H}G_{\epsilon}(x)|^{N-2}\nabla_{H}G_{\epsilon}(x)=\left(-g^{\prime}(F_{\epsilon}(x))\right)^{N-1}
\left( \frac{|x'|^{N-2}x'}{(|x'|^{2}+\epsilon^{2})^{N-1}}\right)
$$
and since $g^{\prime}(t)<0,$ with $\mathcal{L}_{N}$ as in \eqref{pLap},
$$
\mathcal{L}_{N}G_{\epsilon}(x)={\rm div}_{H}(|\nabla_{H}G_{\epsilon}(x)|^{N-2}\nabla_{H}G_{\epsilon}(x))
$$
$$
=(N-1)\left(-g^{\prime}(F_{\epsilon}(x))\right)^{N-2}
g^{\prime\prime}(F_{\epsilon}(x)) \frac{|x'|^{N}}{(|x'|^{2}+\epsilon^{2})^{N}}
$$
$$
+(N-1)\left(-g^{\prime}(F_{\epsilon}(x))\right)^{N-1}
 \frac{2\epsilon^{2}|x'|^{N-2}}{(|x'|^{2}+\epsilon^{2})^{N}}.
$$
The divergence theorem gives
\begin{multline}\label{diveq}
\int_{\Omega}|f|^{N}\mathcal{L}_{N}G_{\epsilon}(x)dx=\int_{\Omega}|f|^{N}
{\rm div}_{H}(|\nabla_{H}G_{\epsilon}(x)|^{N-2}\nabla_{H}G_{\epsilon}(x))dx
\\ =-\int_{\Omega} \nabla_{H}|f|^{N}\cdot (|\nabla_{H}G_{\epsilon}(x)|^{N-2}\nabla_{H}G_{\epsilon}(x))dx.
\end{multline}
Now on the one hand,

\begin{multline}\label{diveqLHS}
\int_{\Omega}|f|^{N}\mathcal{L}_{N}G_{\epsilon}(x)dx
\\
=(N-1)\int_{\Omega}|f|^{N}\left(-g^{\prime}(F_{\epsilon}(x))\right)^{N-2}
g^{\prime\prime}(F_{\epsilon}(x)) \frac{|x'|^{N}}{(|x'|^{2}+\epsilon^{2})^{N}}dx
\\
+(N-1)\int_{\Omega}|f|^{N}\left(-g^{\prime}(F_{\epsilon}(x))\right)^{N-1}
 \frac{2\epsilon^{2}|x'|^{N-2}}{(|x'|^{2}+\epsilon^{2})^{N}}dx
\\ \geq (N-1)\int_{\Omega}|f|^{N}\left(-g^{\prime}(F_{\epsilon}(x))\right)^{N-2}
g^{\prime\prime}(F_{\epsilon}(x)) \frac{|x'|^{N}}{(|x'|^{2}+\epsilon^{2})^{N}}dx.
\end{multline}
On the other hand, using
\begin{multline}\label{diveqRHS}
\left|-\int_{\Omega} \nabla_{H}|f|^{N}\cdot(|\nabla_{H}G_{\epsilon}(x)|^{N-2}\nabla_{H}G_{\epsilon}(x)) dx\right|
\\= \left|N\int_{\Omega}f|f|^{N-2} \nabla_{H}f\cdot (|\nabla_{H}F_{\epsilon}(x)|^{N-2}\nabla_{H}F_{\epsilon}(x)) dx\right|
\\=\left|N\int_{\Omega}|f(x)|^{N-2}f(x)\left(-g^{\prime}(F_{\epsilon}(x))\right)^{N-1}
\left( \frac{|x'|^{N-2} x'\cdot\nabla_{H}f}{(|x'|^{2}+\epsilon^{2})^{N-1}}\right) dx\right|
\\
=N\int_{\Omega}|f(x)|^{N-1}\left(-g^{\prime}(F_{\epsilon}(x))\right)^{N-1}
\left( \frac{|x'|^{N-2}\left| x'\cdot\nabla_{H}f\right|}{(|x'|^{2}+\epsilon^{2})^{N-1}}\right) dx
\\
\leq N\left(
\int_{\Omega}\left(-g^{\prime}(F_{\epsilon}(x))\right)^{N-2}
g^{\prime\prime}(F_{\epsilon}(x))\frac{|x'|^{N}|f(x)|^{N}}{(|x'|^{2}
+\epsilon^{2})^{N}}dx\right)^{\frac{N-1}{N}}
\\ \left(\int_{\Omega}\left(-g^{\prime}(F_{\epsilon}(x))\right)^{2(N-1)}
\left(g^{\prime\prime}(F_{\epsilon}(x))\right)^{-(N-1)}
 \left|\frac{x'}{|x'|}\cdot\nabla_{H}f
 \right|^{N} dx\right)^{\frac{1}{N}}.
\end{multline}
Combining all \eqref{diveq}, \eqref{diveqLHS} and \eqref{diveqRHS} we arrive at
\begin{multline}
(N-1)\int_{\Omega}|f|^{N}\left(-g^{\prime}(F_{\epsilon}(x))\right)^{N-2}
g^{\prime\prime}(F_{\epsilon}(x)) \frac{|x'|^{N}}{(|x'|^{2}+\epsilon^{2})^{N}}dx
\\
\leq N\left(
\int_{\Omega}\left(-g^{\prime}(F_{\epsilon}(x))\right)^{N-2}
g^{\prime\prime}(F_{\epsilon}(x))\frac{|x'|^{N}|f(x)|^{N}}{(|x'|^{2}
+\epsilon^{2})^{N}}dx\right)^{\frac{N-1}{N}}
\\ \left(\int_{\Omega}\left(-g^{\prime}(F_{\epsilon}(x))\right)^{2(N-1)}
\left(g^{\prime\prime}(F_{\epsilon}(x))\right)^{-(N-1)}
 \left| \frac{x'}{|x'|}\cdot\nabla_{H}f\right|^{N} dx\right)^{\frac{1}{N}},
\end{multline}
that is,
\begin{multline}
\left(\frac{N-1}{N}\right)^{N}\int_{\Omega}|f|^{N}\left(-g^{\prime}(F_{\epsilon}(x))\right)^{N-2}
g^{\prime\prime}(F_{\epsilon}(x)) \frac{|x'|^{N}}{(|x'|^{2}+\epsilon^{2})^{N}}dx.
\\
\leq \int_{\Omega}\left(-g^{\prime}(F_{\epsilon}(x))\right)^{2(N-1)}
\left(g^{\prime\prime}(F_{\epsilon}(x))\right)^{-(N-1)}
 \left|\frac{x'}{|x'|}\cdot\nabla_{H}f\right|^{N} dx,
\end{multline}
Now letting $\epsilon\rightarrow 0$ we obtain \eqref{lem3}.
\end{proof}

\section{Horizontal Hardy-Rellich type inequalities and weighted versions}
\label{Sec5}
\subsection{Hardy-Rellich type inequalities} We prove the following Hardy-Rellich type inequalities on the stratified group $\mathbb{G}$:

\begin{thm}\label{LpHR}
Let $\mathbb{G}$ be a stratified group
with $N$ being the dimension of the first stratum, and let $1<p<N$ with $\frac{1}{p}+\frac{1}{q}=1$ and $\alpha,\,\beta\in \mathbb{R}$ be such that $$\frac{p-N}{p-1}\leq\gamma:=\alpha+\beta+1\leq 0.$$
Then for all $f\in C^{\infty}_{0}(\mathbb{G}\backslash\{x'=0\})$ we have
\begin{equation}\label{eqLpHR}
\frac{N+\gamma(p-1)-p}{p}
\left\|\frac{\nabla_{H}f}{|x'|^{\frac{\gamma}{p}}}\right\|^{p}_{L^{p}(\mathbb{G})}\leq
\left\|\frac{1}{|x'|^{\alpha}} \mathcal{L}_{p} f\right\|_{L^{p}(\mathbb{G})}\left\|\frac{\nabla_{H} f}{|x'|^{\beta}}\right\|_{L^{q}(\mathbb{G})},
\end{equation}
where $|\cdot|$ is the Euclidean norm on $\mathbb{R}^{N}$ and $\mathcal{L}_{p}$ is the $p$-sub-Laplacian operator defined by \eqref{pLap}.
\end{thm}

\begin{cor}
When $\beta=0,\,\alpha=-1$ and $q=\frac{p}{p-1}$, the inequality \eqref{eqLpHR} gives a  stratified group Rellich type inequality for $\mathcal{L}_{p}$:
\begin{equation}\label{6.3}
\left\|\nabla_{H} f
\right\|_{L^{p}(\mathbb{G})}\leq\frac{p}{N-p}\left\||x'| \mathcal{L}_{p} f\right\|_{L^{p}(\mathbb{G})},\;1<p<N,
\end{equation}
for all $f\in C_{0}^{\infty}(\mathbb{G}\backslash\{x'=0\}).$
\end{cor}

\begin{cor}
When $\alpha=0,\,\beta=-1$, the inequality \eqref{eqLpHR} implies the following Heisenberg-Pauli-Weyl type  uncertainly principle for $\mathcal{L}_{p},\;1<p<N$:
for each $f\in C^{\infty}_{0}(\mathbb{G}\backslash\{x'=0\})$ we have
\begin{equation}\label{UP1p}
\left\|\nabla_{H}f\right\|^{p}_{L^{p}(\mathbb{G})}
\leq\frac{p}{N-p}\left\| \mathcal{L}_{p}f\right\|_{L^{p}(\mathbb{G})}\left\||x'| \nabla_{H}f\right\|_{L^{q}(\mathbb{G})},\;\frac{1}{p}+\frac{1}{q}=1.
\end{equation}
\end{cor}

\begin{proof}[Proof of Theorem \ref{LpHR}]
As in the proof of Theorem \ref{LpCKN} we have
\begin{multline}\label{pro1}
\int_{\mathbb{G}}
\frac{|\nabla_{H}f(x)|^{p}}
{|x'|^{\gamma}}dx
 =\frac{1}{N-\gamma}\int_{\mathbb{G}}|\nabla_{H}f(x)|^{p} {\rm div}_{H}\left(\frac{x'}{|x'|^{\gamma}}\right)dx
\\  =-\frac{1}{N-\gamma}\int_{\mathbb{G}} \frac{p}{2}
|\nabla_{H}f(x)|^{p-2}\frac{ x'\cdot\nabla_{H}|\nabla_{H}f(x)|^{2}}{|x'|^{\gamma}}dx
\\
=\frac{p}{2(\gamma-N)} \int_{\mathbb{G}}
|\nabla_{H}f(x)|^{p-2}\frac{x'\cdot\nabla_{H}|\nabla_{H}f(x)|^{2}}{|x'|^{\gamma}}dx.
\end{multline}

We also have

\begin{align*}
\int_{\mathbb{G}} \frac{\mathcal{L}_{p}f}{|x'|^{\gamma}}
 x'\cdot\nabla_{H}f(x) dx
& = \int_{\mathbb{G}}
\frac{{\rm div}_{H}(|\nabla_{H}f(x)|^{p-2}\nabla_{H}f(x))}{|x'|^{\gamma}} x'\cdot\nabla_{H}f(x)dx
\\ & =-\int_{\mathbb{G}}
|\nabla_{H}f(x)|^{p-2}\nabla_{H}f(x)\cdot\nabla_{H}\left(\frac{ x'\cdot\nabla_{H}f(x)}{|x'|^{\gamma}}\right) dx
\end{align*}
$$
=-\int_{\mathbb{G}}
|\nabla_{H}f(x)|^{p-2}\left(\frac{|\nabla_{H}f(x)|^{2}}
{|x'|^{\gamma}}+ \frac{ x'\cdot\nabla_{H}|\nabla_{H}f(x)|^{2}}{2|x'|^{\gamma}}-\frac{\gamma\left| x'\cdot\nabla_{H}f(x)\right|^{2} }{|x'|^{\gamma+2}}\right) dx,
$$
that is,
$$
\int_{\mathbb{G}}\frac{|\nabla_{H}f(x)|^{p-2}}{|x'|^{\gamma}}  x'\cdot\nabla_{H}|\nabla_{H}f(x)|^{2}dx
$$
$$
=2\gamma\int_{\mathbb{G}}|\nabla_{H}f(x)|^{p-2}\frac{\left| x'\cdot\nabla_{H}f(x)\right|^{2}}{|x'|^{\gamma+2}}dx-2\int_{\mathbb{G}}
\frac{|\nabla_{H}f(x)|^{p}}
{|x'|^{\gamma}}dx
$$
$$- 2\int_{\mathbb{G}}\frac{\mathcal{L}_{p}f}{|x'|^{\gamma}}
 x'\cdot\nabla_{H}f(x)dx.
$$
Putting this in the right hand side of \eqref{pro1} we obtain
$$
\int_{\mathbb{G}}
\frac{|\nabla_{H}f(x)|^{p}}
{|x'|^{\gamma}}dx
=\frac{p\gamma}{\gamma-N}\int_{\mathbb{G}}|\nabla_{H}f(x)|^{p-2}\frac{\left| x'\cdot\nabla_{H}f(x)\right|^{2}}{|x'|^{\gamma+2}}dx$$
$$-\frac{p}{\gamma-N}\int_{\mathbb{G}}
\frac{|\nabla_{H}f(x)|^{p}}
{|x'|^{\gamma}}dx
- \frac{p}{\gamma-N}\int_{\mathbb{G}}\frac{\mathcal{L}_{p}f}{|x'|^{\gamma}}
 x'\cdot\nabla_{H}f(x) dx.
$$
Thus,
\begin{multline}\label{pro2}
\int_{\mathbb{G}}\frac{\mathcal{L}_{p}f}{|x'|^{\gamma}}
 x'\cdot\nabla_{H}f(x)dx
\\=\frac{N-p-\gamma}{p}\int_{\mathbb{G}}
\frac{|\nabla_{H}f(x)|^{p}}
{|x'|^{\gamma}}dx
+\gamma\int_{\mathbb{G}}|\nabla_{H}f(x)|^{p-2}\frac{\left| x'\cdot\nabla_{H}f(x)\right|^{2}}{|x'|^{\gamma+2}}dx.
\end{multline}
Since $\gamma\leq 0,$ using Schwarz's inequality to the last integrants we get
\begin{multline}\label{pro3}
\int_{\mathbb{G}}\frac{\mathcal{L}_{p}f}{|x'|^{\gamma}}
 x'\cdot\nabla_{H}f(x) dx
\\=\frac{N-p-\gamma}{p}\int_{\mathbb{G}}
\frac{|\nabla_{H}f(x)|^{p}}
{|x'|^{\gamma}}dx
+\gamma\int_{\mathbb{G}}|\nabla_{H}f(x)|^{p-2}\frac{\left| x'\cdot\nabla_{H}f(x)\right|^{2}}{|x'|^{\gamma+2}}dx
\\\geq
\frac{N-p-\gamma}{p}\int_{\mathbb{G}}
\frac{|\nabla_{H}f(x)|^{p}}
{|x'|^{\gamma}}dx
+\gamma\int_{\mathbb{G}}\frac{|\nabla_{H}f(x)|^{p}}{|x'|^{\gamma}}dx
\\
=\frac{N+\gamma(p-1)-p}{p}\int_{\mathbb{G}}\frac{|\nabla_{H}f(x)|^{p}}{|x'|^{\gamma}}dx.
\end{multline}
On the other hand, again using Schwarz's inequality and H\"older's inequality we have
\begin{multline}\label{pro4}
\int_{\mathbb{G}}\frac{\mathcal{L}_{p}f}{|x'|^{\gamma}}
 x'\cdot\nabla_{H}f(x)dx
 \leq \int_{\mathbb{G}}\frac{\mathcal{L}_{p}f}{|x'|^{\gamma-1}}
\left|\nabla_{H}f(x)\right|dx
\\\leq \left(\int_{\mathbb{G}}\left|\frac{\mathcal{L}_{p} f}{|x'|^{\alpha}} \right|^{p}dx\right)^{\frac{1}{p}}\left(\int_{\mathbb{G}}\left|\frac{\nabla_{H} f}{|x'|^{\beta}}\right|^{q}dx\right)^{\frac{1}{q}}.
\end{multline}
Combining it with \eqref{pro3} we prove Theorem \ref{LpHR}.
\end{proof}

\subsection{Weighted versions}
To give an idea for obtaining more general improved weighted Hardy type inequalities let us conclude this paper with the following very short discussion of techniques from  \cite{Yen16} (see also \cite{FS08} and \cite{Ruzhansky-Suragan:Hardy}), now in the setting of stratified groups.

Consider the following weighted ${p}$-sub-Laplacian
\begin{equation}\label{wsubL}
\mathcal{L}_{p,\rho}f={\rm div}_{H}\left(\rho(x)|\nabla_{H}f|^{p-2}\nabla_{H}f\right), \quad 1<p<\infty,
\end{equation}
where $0\leq\rho\in C^{1}(\mathbb{G})$.
\begin{thm}\label{weightedthm1} Let $2\leq p<\infty$.
Let $0\leq F\in C^{\infty}(\mathbb{G})$ and $0\leq\eta\in L^{1}_{loc}(\mathbb{G})$ be such that
\begin{equation}\label{weighteddiv}
\eta F^{p-1}\leq-\mathcal{L}_{p,\rho}F,\quad {\rm a.e.\; in}\;\mathbb{G}.
\end{equation}
Then we have
\begin{equation}\label{weightedthm1ineqn}
\|\eta^{\frac{1}{p}}f\|^{p}_{L^{p}(\mathbb{G})}+C_{p}\left\|
\rho^{\frac{1}{p}}F\nabla_{H}\frac{f}{F}\right\|^{p}_{L^{p}(\mathbb{G})}
\leq\|\rho^{\frac{1}{p}}\nabla_{H}f\|^{p}_{L^{p}(\mathbb{G})},
\end{equation}
for all real-valued functions $f\in C^{\infty}_{0}(\mathbb{G})$. Here $C_{p}$ is a positive constant.
\end{thm}

\begin{proof}[Proof of Theorem \ref{weightedthm1}]
For all $x,y\in \mathbb{R}^{n}$ there exists a positive number $C_{p}$ such that
\begin{equation}\label{anyxy}
|x|^{p}+ C_{p}|y|^{p}+p|x|^{p-2} x\cdot y\leq |x+y|^{p},\quad 2\leq p<\infty.
\end{equation}
Thus,
\begin{multline} |g|^{p}|\nabla_{H}F|^{p}+C_{p}F^{p}|\nabla_{H}g|^{p}+F|\nabla_{H}F|^{p-2}\nabla_{H}F\cdot\nabla_{H}|g|^{p}
\\\leq|g\nabla_{H}F+F\nabla_{H}g|^{p}=|\nabla_{H}f|^{p},
\end{multline}
where $g=\frac{f}{F}.$
It follows that
\begin{align*}
\int_{\mathbb{G}}\rho(x)|\nabla_{H}f(x)|^{p}dx & \geq \int_{\mathbb{G}}\rho(x)|\nabla_{H}F(x)|^{p}|g(x)|^{p}dx  \\
& +C_{p}\int_{\mathbb{G}}\rho(x)|\nabla_{H}g(x)|^{p}|F(x)|^{p}dx \\
& -\int_{\mathbb{G}}{\rm div}_{H}(\rho(x)F(x)|\nabla_{H}F(x)|^{p-2}\nabla_{H}F(x))|g(x)|^{p}dx  \\
&  \geq C_{p}\int_{\mathbb{G}}\rho(x)|\nabla_{H}g(x)|^{p}|F(x)|^{p}dx \\
& +\int_{\mathbb{G}}{\rm -div}_{H}(\rho(x)|\nabla_{H}F(x)|^{p-2}\nabla_{H}F(x))F(x)|g(x)|^{p}dx.
\end{align*}
Using \eqref{weighteddiv} this implies that
\begin{equation}
\int_{\mathbb{G}}\eta(x)|g(x)|^{p}|F(x)|^{p}dx+ C_{p}\int_{\mathbb{G}}\rho(x)|\nabla_{H}g(x)|^{p}|F(x)|^{p}dx\leq\int_{\mathbb{G}}\rho(x)|\nabla_{H}f(x)|^{p}dx.
\end{equation}
Since $g=\frac{f}{F}$ we arrive at
\begin{equation}
\|\eta^{\frac{1}{p}}f\|^{p}_{L^{p}(\mathbb{G})}+C_{p}
\left\|\rho^{\frac{1}{p}}F\nabla_{H}\frac{f}{F}\right\|^{p}_{L^{p}(\mathbb{G})}
\leq\|\rho^{\frac{1}{p}}\nabla_{H}f\|^{p}_{L^{p}(\mathbb{G})},
\end{equation}
proving \eqref{weightedthm1ineqn}.
\end{proof}

\begin{rem}\label{remw1}
For $p=2$ there is equality in \eqref{anyxy} with $C_{2}=1$, that is, the above proof gives the following remainder formula
\begin{equation}\label{p=2}
\left\|
\rho^{\frac{1}{2}}F\nabla_{H}\frac{f}{F}\right\|^{2}_{L^{2}(\mathbb{G})}=
\|\rho^{\frac{1}{2}}\nabla_{H}f\|^{2}_{L^{2}(\mathbb{G})}-\|\eta^{\frac{1}{2}}f\|^{2}_{L^{2}(\mathbb{G})}.
\end{equation}
\end{rem}

\begin{rem}\label{remw2}
For $1<p<2$ the inequality \eqref{anyxy} can be stated as
for all $x,y\in \mathbb{R}^{n}$ there exists a positive number $C_{p}$ (see e.g. \cite[Lemma 4.2]{Lind:p-ineq}) such that
\begin{equation}
|x|^{p}+ C_{p}\frac{|y|^{p}}{(|x|+|y|)^{2-p}}+p|x|^{p-2}x\cdot y\leq |x+y|^{p},\quad 1<p<2.
\end{equation}
In turn, from the proof it follows that
\begin{equation}\label{p1-2}
\|\eta^{\frac{1}{p}}f\|^{p}_{L^{p}(\mathbb{G})}
+C_{p}\left\|\rho^{\frac{1}{2}}\left(\left|\frac{f}{F}\nabla_{H}F\right|
+F\left|\nabla_{H}\frac{f}{F}\right|\right)^{\frac{p-2}{2}}|F|\nabla_{H}\frac{f}{F}\right\|^{2}_{L^{2}(\mathbb{G})}
\leq\|\rho^{\frac{1}{p}}\nabla_{H}f\|^{p}_{L^{p}(\mathbb{G})}
\end{equation}
for all real-valued $f\in C^{\infty}_{0}(\mathbb{G})$.
\end{rem}

\begin{prop}\label{thetacor}
For $f\in C_{0}^{\infty}(\mathbb{G})$ we have
\begin{equation}\label{coreq} \left\|\frac{f}{|x'|}\right\|_{L^{p}}\leq\frac{p}{\theta-p-2} \left\|\nabla_{H}f\right\|_{L^{p}},\quad 1<p<\theta-2,\;\theta\leq 2+N,\; \theta\in\mathbb{R}.
\end{equation}
\end{prop}
\begin{proof}[Proof of Proposition \ref{thetacor}.]
In Theorem \ref{weightedthm1} taking $\rho=1$ and $$F_{\epsilon}=|x'_{\epsilon}|^{-\frac{\theta-p-2}{p}}
=\left((x'_{1}+\epsilon)^{2}+\ldots+(x'_{n}+\epsilon)^{2}\right)^{-\frac{\theta-p-2}{2p}}, $$
for a given $\epsilon>0,$ using the identity \eqref{gradgamma} we get
\begin{multline}\label{thmcond}
-\mathcal{L}_{p,1}F_{\epsilon} =-{\rm div}_{H}\left(|\nabla_{H}F_{\epsilon}|^{p-2}\nabla_{H}F_{\epsilon}\right) \\ =
-{\rm div}_{H}\left(|\nabla_{H}|x'_{\epsilon}|^{-\frac{\theta-p-2}{p}}|^{p-2}\nabla_{H}|x'_{\epsilon}|^{-\frac{\theta-p-2}{p}}\right)
\\
=\frac{\theta-p-2}{p}\left|\frac{\theta-p-2}{p}\right|^{p-2} \left(\frac{\theta-p-2}{p}-\theta+2+N\right)|x'_{\epsilon}|^{-\frac{(\theta-p-2)(p-1)}{p}-p}
\\
=\left(\left|\frac{\theta-p-2}{p}\right|^{p}+\frac{\theta-p-2}{p}\left|\frac{\theta-p-2}{p}\right|^{p-2}(-\theta+2+N)\right)
|x'_{\epsilon}|^{-\frac{(\theta-p-2)(p-1)}{p}-p}.
\end{multline}
If $1<p<\theta-2$ and $\theta\leq 2+N$, then
\eqref{thmcond} gives
\begin{equation}
-\mathcal{L}_{p,1}F_{\epsilon}\geq \left|\frac{\theta-p-2}{p}\right|^{p}\frac{1}{|x'_{\epsilon}|^{p}}F_{\epsilon}^{p-1},
\end{equation}
that is, according to the assumption in Theorem \ref{weightedthm1}, we can put
$$\eta(x)=\left|\frac{\theta-p-2}{p}\right|^{p}\frac{1}{|x'_{\epsilon}|^{p}}.$$
It shows that \eqref{weightedthm1ineqn} (and \eqref{p1-2}) implies \eqref{coreq}.
\end{proof}

Note that in the case of the Heisenberg group \eqref{coreq} was proved by D'Ambrosio in \cite{DAmbrosio-Difur04}. Here it is worth to recall that on the Heisenberg group
we have $Q=2+N.$

\end{document}